\documentclass[12pt]{amsart}

\usepackage[colorlinks,
linkcolor=blue,
anchorcolor=blue,
citecolor=blue]{hyperref}
\usepackage{epsfig}
\usepackage{graphicx}
\usepackage{amssymb, amstext, amscd, amsmath}
\usepackage{amsthm, mathrsfs, amsfonts,dsfont}
\usepackage{fullpage}
\usepackage{txfonts}
\usepackage{fancybox}
\usepackage{color}
\usepackage{cite}
\usepackage{comment}
\usepackage{tikz-cd} %»æÖÆ½»»»Í¼,Ó³Éä.

\allowdisplaybreaks

\newtheorem{theorem}{Theorem}[section]
\newtheorem{lemma}{Lemma}[section]

\newtheorem{corollary}{Corollary}[section]

\theoremstyle{definition}

\newtheorem{remark}{Remark}[section]

\numberwithin{equation}{section}

\begin{document}
\title {Block Repetition of Numerical Invariants for the Submodules $[z^k-w^k]$ in $H^2(\mathbb D^2)$}
\author{Yin Liu}
\address{School of Mathematics and Statistics, Nanyang Normal University,
Nanyang, Henan, 473061, P. R. China} 
\email{lylight@mail.bnu.edu.cn}

\author{Yufeng Lu}
\address{School of Mathematical Sciences, Dalian University of Technology,
Dalian, Liaoning, 116024, P. R. China}
\email{lyfdlut@dlut.edu.cn}

\author{Chao Zu*}
\address{School of Mathematical Sciences, Dalian University of Technology,
Dalian, Liaoning, 116024, P. R. China}
\email{zuchao@mail.dlut.edu.cn}

\thanks{*Corresponding author.}
%\thanks{This research is supported by National Natural Science Foundation of China
%(No. 12031002).}

\begin{abstract}
For $k\ge 2$, let $M_k=[z^k-w^k]$ be the principal homogeneous submodule of the Hardy space over the bidisk. We determine Yang's complete sequence of numerical invariants and prove

$$
\Sigma_0(M_k)=\frac{\pi^2}{6},\qquad
\Sigma_j(M_k)=\Sigma_{\lceil j/k\rceil}([z-w]),\quad j\ge1.
$$
The proof exploits a residue-class decomposition of the Toeplitz matrices associated with the graded wandering spaces. Consequently, Yang's monotonicity conjecture holds for the family $\{M_k:k\ge2\}$. We also show that the nonzero spectral data of the core operator are independent of $k$, whereas the numerical invariant sequence recovers $k$ from the length of its constant blocks. Thus the higher numerical invariants detect module-theoretic information invisible to the core spectrum.

\end{abstract}

\maketitle

%\section{Introduction}

%For certain classes of homogeneous submodules, these numerical invariants can be computed explicitly.
%In particular, the case of the submodule generated by the linear polynomial $(z-w)$ has been studied in detail, and explicit formulas for the corresponding invariants are available.
%These computations also verify the monotonicity property of the invariants in this setting.
%
%The purpose of the present paper is to carry out a detailed computation of the numerical invariants for the homogeneous submodule generated by $(z-w)^2$.
%This example represents a natural next step beyond the linear case and requires more involved calculations.
%We obtain explicit expressions for the invariants and verify the monotonicity property in this concrete situation.
%Our results provide a complete analysis of this example and complement the existing literature on numerical invariants of homogeneous submodules.

\section{Introduction}

Let $\mathbb{D}^2 = \{(z, w) \in \mathbb{C}^2 : |z| < 1, |w| < 1\}$ be the unit bidisk, and let $H^2(\mathbb{D}^2)$ denote the Hardy space on $\mathbb{D}^2$, regarded as a Hilbert module over the polynomial ring $\mathbb{C}[z, w]$. A closed subspace $M \subseteq H^2(\mathbb{D}^2)$ is called a \emph{submodule} if it is invariant under multiplication by each of the coordinate functions $z$ and $w$. Thus the Hardy module on the bidisk provides a natural setting for studying commuting shift operators and their invariant subspaces from an operator-theoretic viewpoint.

Unlike the one-variable Hardy space $H^2(\mathbb{D})$, where Beurling’s theorem gives a complete description of all submodules in terms of inner functions \cite{Beu}, the submodule structure of $H^2(\mathbb{D}^2)$ is far more complex. Even for submodules generated by simple polynomials, there is no comparable complete Beurling-type classification. For instance, the submodule $[z - w]$ generated by $z - w$ cannot be expressed as $\theta H^2(\mathbb{D}^2)$ for any two-variable inner function $\theta$. The absence of a canonical model makes the classification and analysis of bidisk submodules a challenging task (see, e.g., \cite{Che,Dou,Rud}).

A productive line of inquiry has been to study submodules of $H^2(\mathbb{D}^2)$ via operator-theoretic invariants derived from the module action. Two natural pairs of operators are associated with any submodule $M$: the compression pair $(S_1, S_2)$ acting on $M^\perp$, which is canonically identified with the quotient module $H^2(\mathbb D^2)/M$, and the restriction pair $(R_1, R_2)$ acting on $M$ itself. These pairs encode structural information about $M$ and have been the subject of extensive investigation \cite{Core,Ya3,Ya1,Ya2,Ya2004,Ya2005,Survey}.

Motivated by the analysis of the core operator associated with $(R_1, R_2)$, R. Yang \cite{Ya3} introduced a sequence of numerical invariants
\[
\{\Sigma_j(M)\}_{j\geq 0}.
\]
Yang first introduced $\Sigma_0$ and $\Sigma_1$ through Hilbert--Schmidt norms of commutators and subsequently defined higher-order invariants using the graded wandering-space bases. These invariants have proven useful for distinguishing submodules and for quantifying fine differences in their operator-theoretic behavior. However, explicit computations are available only in a few special cases. For example, when $M = H^2(\mathbb{D}^2)$, one has $\Sigma_0 = 1$ and $\Sigma_j = 0$ for all $j \geq 1$, whereas for the submodule $[z-w]$, the sequence $\{\Sigma_j\}$ admits a nontrivial explicit expression $\{\frac{1}{6} \pi^2,\frac{1}{6} \pi^2-1,\frac{5}{6} \pi^2-8,\frac{13}{6} \pi^2-\frac{85}{4}, \cdots\}$ \cite{Ya3}.

Motivated by these examples, Yang proposed the following conjecture:

\medskip
\noindent\textbf{Conjecture.}
\emph{For every submodule $M\subset H^2(\mathbb D^2)$, the sequence $\{\Sigma_j(M)\}_{j\geq 0}$ is nonincreasing.}
\medskip

At present, the conjecture remains open in this generality. It is also open for finitely generated polynomial submodules in general. For some homogeneous submodules, partial results have already been obtained; see \cite{Fat, Liu}. In \cite{Liu}, for the submodule $[z-w]$, it was shown that the sequence $\{ \Sigma_j\}$ is given by
\[\Sigma_j=\sum_{m=0}^{\infty} \frac{(m+1)^2}{(m+j)^2(m+j+1)^2},\,\,\,j \geq 1,\]
and $\Sigma_0=\frac{\pi^2}{6}$.

Nevertheless, concrete verifications of the conjecture for explicit, nontrivial examples are still scarce. This paper aims to study the homogeneous polynomial submodule $[z^k-w^k]~~~(k \geq 2)$. The special form of the polynomial \(z^k-w^k\) leads to an explicit
residue-class decomposition of the associated Toeplitz matrices, but the resulting family is not spectrally trivial. The associated Toeplitz matrices split into $k$ tridiagonal chains, and the interaction of these chains with the shifted wandering vectors gives rise to a nontrivial repetition pattern in the higher numerical invariants.

Our first result gives explicit determinant and cofactor formulas for
the Toeplitz matrices associated with \(M_k\). The essential mechanism
is a decomposition into residue classes modulo \(k\), under which the
matrices become direct sums of standard tridiagonal blocks. This
decomposition also allows us to determine the spectrum of the
associated core operator.

Our main result gives a complete description of the numerical invariants
of the submodules \(M_k=[z^k-w^k]\).

\begin{theorem}\label{thm:main}
Let \(M_k=[z^k-w^k]\), where \(k\geq 2\). Then
\[
\Sigma_0(M_k)=\frac{\pi^2}{6},
\qquad
\Sigma_j(M_k)
=
\Sigma_{\left\lceil j/k\right\rceil}([z-w]),
\qquad j\geq 1.
\]
\end{theorem}

In particular, for every \(m\geq 1\),
\[
\Sigma_{km-k+1}(M_k)
=
\Sigma_{km-k+2}(M_k)
=
\cdots
=
\Sigma_{km}(M_k)
=
\Sigma_m([z-w]).
\]
Thus each numerical invariant of the linear model \([z-w]\) is
repeated in a constant block of length \(k\). Since the sequence
\(\{\Sigma_m([z-w])\}_{m\geq 1}\) is strictly decreasing, it follows
that \(\{\Sigma_j(M_k)\}_{j\geq 0}\) is nonincreasing, with equality
within each block and strict decrease between consecutive blocks.
Consequently, Yang's monotonicity conjecture holds for the entire
family \(\{M_k:k\geq 2\}\).

An additional consequence is that the higher numerical invariants
distinguish the members of the family \(\{M_k\}_{k\geq 2}\).
Indeed, Theorem~\ref{thm:main} gives
\[
\Sigma_1(M_k)
=
\Sigma_2(M_k)
=
\cdots
=
\Sigma_k(M_k)
=
\Sigma_1([z-w]),
\]
whereas
\[
\Sigma_{k+1}(M_k)
=
\Sigma_2([z-w])
<
\Sigma_1([z-w]).
\]
Hence the length of the first constant block is exactly \(k\), and
therefore
\[
k
=
\max\bigl\{
j\geq 1:
\Sigma_j(M_k)=\Sigma_1(M_k)
\bigr\}.
\]
This contrasts with the associated core operators: their nonzero
eigenvalues, together with multiplicities, are independent of \(k\).
More precisely, for every \(k\geq 2\), the nonzero eigenvalues are
\[
\{1\}
\cup
\left\{
\pm\frac{1}{m+1}:m\geq 1
\right\},
\]
with multiplicity one. Thus the spectral data of the core
operator do not determine \(k\), whereas the full numerical invariant
sequence does. Hence the higher numerical invariants retain
module-theoretic information that is not detected by the spectral
data of the core operator.

We emphasize that the family considered here is structurally different
from the submodule \([(z-w)^2]\) studied in \cite{Liu}.
Indeed,
\[
z^k-w^k
=
\prod_{\omega^k=1}(z-\omega w)
\]
has \(k\) distinct linear factors, whereas \((z-w)^2\) has a repeated
factor. This distinction leads to essentially different matrix
structures. For \(z^k-w^k\), the Toeplitz matrices associated with
the graded wandering spaces decompose, according to residue classes
modulo \(k\), into several independent tridiagonal chains. It is this
modulo-\(k\) block decomposition, together with the resulting sparsity
of the cofactors, that produces the block-repetition phenomenon in the
higher numerical invariants. By contrast, the repeated factor in \((z-w)^2\) leads to a
different Gram/Toeplitz structure. Therefore,
the block-repetition formula obtained in the present paper is not a
direct consequence of, or a direct substitution into, the results for
\([(z-w)^2]\).

The remainder of the paper is organized as follows. In Section~\ref{s2}, we
review the basic definitions and preliminary results concerning the
core operator and the numerical invariants of submodules. In
Section~\ref{s3}, we exploit the decomposition into residue classes modulo
\(k\) to compute the determinants and algebraic cofactors of the
associated Toeplitz matrices, and we determine the eigenvalues of the
core operator of \(M_k\). In Section~\ref{s4}, we derive a unified formula for
the inner products
\[
\langle w^j\varphi_n,z^j\psi_n\rangle,
\]
identify their exact nonzero support, and use this formula to prove the
block-repetition theorem and the monotonicity of the complete numerical
invariant sequence. We also show that the full numerical invariant sequence determines
the parameter \(k\), thereby revealing module-theoretic information
that is not encoded in the spectral data of the core operator.

% We also show that the invariant sequence recovers
% the parameter \(k\), thereby exhibiting information that is not visible
% in the spectral data of the core operator.

\section{Preliminaries}\label{s2}

\subsection*{Core operator and numerical invariants}
Let  $K(\lambda,z)=\frac{1}{(1-\overline{\lambda_1}z)(1-\overline{\lambda_2}w)}$ be the reproducing kernel of $H^2(\mathbb{D}^2)$, and denote by $K^M(\lambda,z)$ the reproducing kernel of a submodule $M \subseteq H^2(\mathbb{D}^2)$. The \emph{core function} of $M$ is defined by \[G^M(\lambda,z):=\frac{K^M(\lambda,z)}{K(\lambda,z)}=(1-\overline{\lambda_1}z)(1-\overline{\lambda_2}w)K^M(\lambda,z).\]

The corresponding \emph{core operator} $C^M$ on $H^2(\mathbb{D}^2)$ is given by  
\[ C^M(f)(z):=\int_{\mathbb{T}^2} G^M(\lambda,z) f(\lambda) dm(\lambda),~~~z\in\mathbb{D}^2, \]
where $dm(\lambda)$ is the normalized Lebesgue measure on the torus $\mathbb{T}^2$. When the submodule $M$ is clear from the context, we simply write $C$ for the core operator.

The core function and core operator were introduced by Guo and Yang \cite{Core} and have been extensively studied in \cite{Core,Ya2004,Ya2005}. For every submodule $M \subseteq H^2(\mathbb{D}^2)$, it is known that $G^M(\lambda,\lambda) = 1$ almost everywhere on $\mathbb{T}^2$, and that $C^M$ vanishes on the orthogonal complement $M^\perp = H^2(\mathbb{D}^2) \ominus M$. If $C^M$ is Hilbert–Schmidt, or equivalently if $G^M \in L^2(\mathbb{T}^2 \times \mathbb{T}^2)$, then $M$ is called a Hilbert–Schmidt submodule. The Hilbert–Schmidt property of submodules has been investigated in numerous works; see, e.g., \cite{Luo,Zou,Zu1,Zu2}.

Two fundamental operator pairs associated with a submodule $M \subset H^2(\mathbb{D}^2)$ are $(S_1, S_2) \) and \( (R_1, R_2)$, defined by  
\[S_i f = (I - P) z_i f, \qquad R_i g = z_i g, \qquad i = 1,2,\]
where $f \in H^2(\mathbb{D}^2) \ominus M$, $g \in M$, and $P$ is the orthogonal projection from $H^2(\mathbb{D}^2)$ onto $M$. The pair $(S_1,S_2)$ is a pair of commuting contractions on $H^2(\mathbb{D}^2) \ominus M$, while $(R_1, R_2)$ is a pair of commuting isometries on $M$. These two operator pairs encode a wealth of structural information about $M$ \cite{Core}.

A key relation linking the core operator with $(R_1, R_2)$ is  
\[C^M = I - R_1 R_1^* - R_2 R_2^* + R_1 R_2 R_1^* R_2^*.\]

Moreover, $C^2$ is unitarily equivalent to the diagonal block matrix  
\begin{equation*}
  \left(
    \begin{array}{cc}
      [R_1^*, R_1][R_2^*,R_2][R_1^*,R_1] & 0 \\
      0 & [R_1^*,R_2][R_2^*,R_1] \\
    \end{array}
  \right).
\end{equation*}

Yang introduced two numerical invariants associated with $(R_1, R_2)$ \cite{Ya3}:  
\[ \Sigma_0(M)=\|[R_2^*, R_2][R_1^*,R_1]\|_{H.S.}^2;\,\,\,\,\,\,\,\,\,\Sigma_1(M)=\|[R_1^*, R_2]\|_{H.S.}^2,\]
where $[A,B] = AB - BA$ and $\|\cdot\|_{H.S.}$ denotes the Hilbert–Schmidt norm. When no confusion arises, we write simply $\Sigma_0$ and $\Sigma_1$. If $M$ is Hilbert–Schmidt, then  
$$\Sigma_0 - \Sigma_1=1,~~~\|C\|_{H.S.}^2=\Sigma_0 + \Sigma_1.$$

We shall use the following lemma of Yang.

\begin{lemma}[\cite{Ya3}, Lemma 3.2]\label{l2.1}
If $\{\phi_n:n\geq 0\}$ is an orthonormal basis for $M\ominus zM$ and $\{\psi_n:n\geq 0\}$ is an orthonormal basis for $M\ominus wM$, then
\[ \Sigma_0=\sum_{n=0}^{\infty}|\langle\phi_n, \psi_n \rangle|^2 ,\,\,\,\,\,\,\,\,\Sigma_1=\sum_{n=0}^{\infty}|\langle w\phi_n, z\psi_n \rangle|^2.\]
\end{lemma}

Motivated by the above lemma, Yang also defines the higher-order numerical invariants in \cite{Ya3}:
\[\Sigma_j=\sum_{n=0}^{\infty}|\langle w^j \phi_n, z^j \psi_n \rangle|^2,\,\,\,\,\,\,j \geq 2.\]

To facilitate the subsequent derivation, we present the following important lemma at the end of this section. This lemma follows from Theorem~1 on p.~131 of \cite{Fis}.

\begin{lemma}\label{l2.2}
Let \(T_m\) be the $m\times m$ tridiagonal matrix
\begin{equation*}
 T_m= \left(
    \begin{array}{ccccccc}
      2 & -1 & 0 & \cdots & 0 & 0 & 0\\
       -1 & 2 & -1 & \cdots & 0 & 0 & 0\\
      0 & -1 & 2 & \cdots & 0 & 0 & 0\\
      \vdots & \vdots & \vdots & \ddots & \vdots & \vdots & \vdots\\
      0 & 0 & 0 & \cdots & 2 & -1 & 0\\
      0 & 0 & 0 & \cdots & -1 & 2 & -1\\
      0 & 0 & 0 & \cdots & 0 & -1 & 2\\
    \end{array}
  \right)_{m \times m}.
\end{equation*}
Then for $1\leq a,b\leq m$,
\[
(T_m^{-1})_{a,b}=\frac{\min(a,b)\,(m+1-\max(a,b))}{m+1}. 
\]
\end{lemma}

\section{Block decomposition for the submodule $M_k$}\label{s3}
Throughout the following sections, empty index ranges are omitted.

The subspaces $M\ominus zM$ and $M\ominus wM$ are sometimes called wandering spaces for the submodule $M$. They capture important information about $M$. In general, explicit orthonormal bases for these wandering spaces are difficult to obtain. For homogeneous submodules, however, the graded orthogonal decomposition makes such computations tractable.

We shall use the following result from \cite{Fat}.

\begin{lemma}[\cite{Fat}, Proposition 2.1]\label{l3.1}
Let $M=[p]$ be the homogeneous submodule generated by a homogeneous polynomial $p$. Then $\{\phi_n:n\geq 0\}$ is an orthonormal basis for $M\ominus zM$ and $\{\psi_n:n\geq 0\}$ is an orthonormal basis for $M\ominus wM$, where
\[\phi_0=\psi_0=\frac{p}{\|p\|},\,\,\,\,\phi_n=\frac{\sum_{j=0}^{n}pA_{0,j}^nz^jw^{n-j}}{\sqrt{D_{n+1}D_n}},\,\,\,\,\psi_n=\frac{\sum_{j=0}^npA_{n,j}^nz^jw^{n-j}}{\sqrt{D_{n+1}D_n}},\,\,\,\,n\geq1.\]
\end{lemma}

In Lemma \ref{l3.1}, $p$ is a homogeneous polynomial of degree $d$ with
\[p=\sum_{j=0}^{d}c_jz^jw^{d-j},\]
and
\begin{equation*}
 A^n= \left(
    \begin{array}{ccccc}
      \|p\|^2 & \overline{\langle pw,pz \rangle} & \overline{\langle pw^2,pz^2 \rangle} & \cdots & \overline{\langle pw^n,pz^n \rangle} \\
      \langle pw,pz \rangle & \|p\|^2 & \overline{\langle pw,pz \rangle} & \cdots & \overline{\langle pw^{n-1},pz^{n-1} \rangle} \\
      \langle pw^2,pz^2 \rangle & \langle pw,pz \rangle & \|p\|^2 & \cdots & \overline{\langle pw^{n-2},pz^{n-2} \rangle} \\
      \vdots & \vdots & \vdots & \ddots & \vdots \\
      \langle pw^n,pz^n \rangle & \langle pw^{n-1},pz^{n-1} \rangle & \langle pw^{n-2},pz^{n-2} \rangle & \cdots & \|p\|^2 \\
    \end{array}
  \right).
\end{equation*}
Note that $A^n$ is an $(n+1) \times (n+1)$ Toeplitz matrix. For convenience, we set $D_n=\textrm{det} \,\, A^{n-1}$ for $ n\geq 1$ and $D_0=1$. Furthermore, we write $A^n$ as $(a_{i,j})_{i,j=0}^n$, where $a_{i,j}=\langle pw^{i-j},pz^{i-j} \rangle, \,\,\,i,j=0,1,\cdots,n$, and denote its cofactor matrix by $(A_{i,j}^n)_{i,j=0}^n$.

\begin{lemma}[\cite{Fat}, Corollary 2.3]\label{l3.2}
For the homogeneous submodule $[p]$, we have
\[\Sigma_0=\sum_{n=0}^{\infty} \left|\frac{A_{0,n}^n}{D_n} \right|^2.\]
\end{lemma}

\begin{lemma}[\cite{Fat}, Theorem 3.7]\label{l3.3}
For the homogeneous submodule $[p]$, the core operator has eigenvalues
\[0,1,\pm \left(1-\frac{(|D_n|^2-|A_{0,n}^n|^2)^2}{D_{n-1}D_n^2D_{n+1}} \right)^{1/2},\,\,\,\,n \geq 1.\]
\end{lemma}

We now specialize to $p=z^k-w^k~~~(k \geq 2)$. A direct computation gives 
\[ \|p\|^2=2;\,\,\,\,\langle pw^j,pz^j \rangle=0,\,\,j=1,2,\cdots,k-1;\,\,\,\,\langle pw^k,pz^k \rangle=-1;\,\,\,\,\langle pw^i,pz^i \rangle=0,\,\,i \geq k+1;\]
and
\begin{equation*}
 A^n= \left(
    \begin{array}{cccccccccccc}
      2 & 0 & \cdots & 0 & -1 & 0 & \cdots & 0 & 0 & \cdots & 0 & 0\\
       0 & 2 & \cdots & 0 & 0 & -1 & \cdots & 0 & 0 & \cdots & 0 & 0\\
       \vdots & \vdots & \vdots & \vdots & \vdots & \vdots & \vdots & \vdots & \vdots & \vdots & \vdots & \vdots\\
      0 & 0 & \cdots & 2 & 0 & 0 & \cdots & 0 & 0 & \cdots& 0 & 0\\
       -1 & 0 & \cdots & 0 & 2 & 0 & \cdots & 0 & 0 & \cdots & 0 & 0\\
       0 & -1 & \cdots & 0 & 0 & 2 & \cdots & 0 & 0 & \cdots & 0 & 0\\
      \vdots & \vdots & \vdots & \vdots & \vdots & \vdots & \ddots & \vdots & \vdots & \cdots & \vdots & \vdots\\
      0 & 0 & 0 & 0 & 0 & 0 & \cdots & 2 & 0 & \cdots & 0 & -1\\
      0 & 0 & 0 & 0 & 0 & 0 & \cdots & 0 & 2 & \cdots & 0 & 0\\
      \vdots & \vdots & \vdots & \vdots & \vdots & \vdots & \vdots & \vdots & \vdots & \vdots & \vdots & \vdots\\
      0 & 0 & 0 & 0 & 0 & 0 & \cdots & 0 & 0 & \cdots & 2 & 0\\
      0 & 0 & 0 & 0 & 0 & 0 & \cdots & -1 & 0 & \cdots & 0 & 2\\
    \end{array}
  \right)_{(n+1) \times (n+1)}.
\end{equation*}

We first compute the determinants $D_n$.

\begin{theorem}\label{t3.1}
Let \(M_k=[z^k-w^k]\), where \(k\geq2\), and let
\(A^{n-1}\) be the \(n\times n\) Toeplitz matrix associated with
\(z^k-w^k\). Write
\[
n=kq+r,\qquad 0\leq r<k.
\]
Then, after a permutation of rows and columns according to the residue
classes modulo \(k\), the matrix \(A^{n-1}\) is permutation similar to
\[
B_0\oplus B_1\oplus\cdots\oplus B_{k-1},
\]
where
\[
B_0,\ldots,B_{r-1}\cong T_{q+1},
\]
and
\[
B_r,\ldots,B_{k-1}\cong T_q .
\]
Consequently,
\[
\det A^{n-1}
=
(q+2)^r(q+1)^{k-r}.
\]
Here \(T_0\) is regarded as the empty matrix and we adopt the
convention
\[
\det T_0=1.
\]
\end{theorem}

\begin{proof}
Recall that the entries of \(A^{n-1}\) are determined by the
coefficients of the polynomial
\[
z^k-w^k .
\]
More precisely,
\[
A^{n-1}=(a_{ij})_{i,j=0}^{n-1},
\]
where
\[
a_{ij}
=
\begin{cases}
2,& i=j,\\
-1,& |i-j|=k,\\
0,& \text{otherwise}.
\end{cases}
\]
Therefore, \(A^{n-1}\) is a banded Toeplitz matrix whose nonzero
entries occur only when the indices differ by \(0\) or by \(k\).

We now rearrange the indices according to their residue classes modulo
\(k\). Let \(P\) be the permutation matrix corresponding to this
rearrangement, where the indices in each residue class remain in their
natural increasing order. Then
\[
PA^{n-1}P^{*}
\]
is block diagonal, because two indices belonging to different residue
classes modulo \(k\) cannot be connected by a nonzero entry of
\(A^{n-1}\).

The residue classes are
\[
0,1,\ldots,k-1.
\]
Since
\[
n=kq+r,\qquad 0\le r<k,
\]
the residue class \(s\) contains the indices
\[
s,\ s+k,\ s+2k,\ldots
\]
which are smaller than \(n\).

For
\[
0\leq s\leq r-1,
\]
there are exactly \(q+1\) such indices, whereas for
\[
r\leq s\leq k-1,
\]
there are exactly \(q\) such indices.

Hence the corresponding diagonal blocks satisfy
\[
B_s\cong T_{q+1},
\qquad 0\leq s\leq r-1,
\]
and
\[
B_s\cong T_q,
\qquad r\leq s\leq k-1.
\]

Since permutation does not change the determinant, we have
\[
\det A^{n-1}
=
\det(PA^{n-1}P^{*}).
\]
Using the block diagonal decomposition, it follows that
\[
\det A^{n-1}
=
\prod_{s=0}^{k-1}\det B_s .
\]

From p.518 of \cite{Fat}, we know that
\[
\det T_m=m+1.
\]
Therefore,
\[
\begin{aligned}
\det A^{n-1}
&=
(\det T_{q+1})^r(\det T_q)^{k-r}\\
&=
(q+2)^r(q+1)^{k-r}.
\end{aligned}
\]

When \(q=0\), the blocks corresponding to \(T_0\) are empty and,
by convention,
\[
\det T_0=1.
\]
Thus the above formula remains valid for all
\[
n=kq+r,\qquad0\le r<k.
\]

This completes the proof.
\end{proof}

\begin{remark}
For the convenience of subsequent derivations, we now expand the conclusion of Theorem 3.1 into an explicit expression.

If $r=0$, then we have $n=kq$ and $q=\frac{n}{k}$, thus, we get that
\[D_n=(\frac{n}{k}+1)^k=\frac{(n+k)^k}{k^k}.\]

Similarly, we have that if $1 \leq r < k$, then $q=\frac{n-r}{k}$ and
\[D_n=(\frac{n-r}{k}+2)^r(\frac{n-r}{k}+1)^{k-r}=\frac{(n-r+2k)^r \cdot (n-r+k)^{k-r}}{k^k}.\]

In summary, it can be obtained that
\[D_n=\frac{(n-l+2k)^l(n-l+k)^{k-l}}{k^k},~~~~ n \equiv l~(\mathrm{mod}~k),\]
where $l=0,1,2,\cdots,k-1$.
\end{remark}

\vspace{0.2cm}

\rm{We next compute the cofactors in the last row of $A^n$.}

\begin{theorem}\label{t3.2}
For the homogeneous submodule \(M_k=[z^k-w^k]\) with \(k\geq 2\),
write
\[
n=kq+r,\qquad 0\leq r<k.
\]
Then
\[
A^n_{n,l}
=
\begin{cases}
\displaystyle
\frac{(n-r+2k)^r(n-r+k)^{k-r-1}(k+l-r)}
{k^k},
& l\equiv r\pmod{k},\\[3mm]
0,
& \text{otherwise},
\end{cases}
\]
where \(0\leq l\leq n\).
\end{theorem}

\begin{proof}
Recall that \(A^n_{n,l}\) denotes the algebraic cofactor of the
entry in row \(n\) and column \(l\) of the matrix \(A^n\).
Since \(A^n\) is positive definite, it is invertible. By the
adjugate formula,
\[
(A^n)^{-1}
=
\frac{\operatorname{adj}(A^n)}{\det A^n}.
\]
Hence
\begin{equation}\label{3.1}
A^n_{n,l}
=
\det(A^n)\bigl((A^n)^{-1}\bigr)_{l,n}.
\end{equation}

Let \(P\) be the permutation matrix which groups the indices
\(0,1,\ldots,n\) according to their residue classes modulo \(k\),
while preserving the increasing order within each residue class.
As in the proof of Theorem~\ref{t3.1},
\[
PA^nP^*
=
B_0\oplus B_1\oplus\cdots\oplus B_{k-1},
\]
where
\[
B_0,\ldots,B_r\cong T_{q+1},
\qquad
B_{r+1},\ldots,B_{k-1}\cong T_q.
\]
Consequently,
\[
P(A^n)^{-1}P^*
=
B_0^{-1}\oplus B_1^{-1}\oplus\cdots\oplus B_{k-1}^{-1}.
\]

It follows immediately that
\[
\bigl((A^n)^{-1}\bigr)_{l,n}=0
\]
unless \(l\) and \(n\) belong to the same residue class modulo \(k\).
Since
\[
n\equiv r\pmod{k},
\]
we obtain
\[
\bigl((A^n)^{-1}\bigr)_{l,n}=0
\qquad\text{if}\qquad
l\not\equiv r\pmod{k}.
\]
Therefore, by \eqref{3.1},
\[
A^n_{n,l}=0
\qquad\text{if}\qquad
l\not\equiv r\pmod{k}.
\]

Now suppose that
\[
l\equiv r\pmod{k}.
\]
Write
\[
l=r+kj,
\qquad
0\leq j\leq q.
\]
The indices belonging to the residue class \(r\) are
\[
r,\ r+k,\ r+2k,\ \ldots,\ r+qk=n.
\]
Thus, inside the block corresponding to the residue class \(r\),
the index \(l=r+kj\) occupies position \(j+1\), while
\(n=r+qk\) occupies position \(q+1\). This block is precisely
\(T_{q+1}\). Hence Lemma~\ref{l2.2} gives
\[
\bigl((A^n)^{-1}\bigr)_{l,n}
=
\bigl(T_{q+1}^{-1}\bigr)_{j+1,q+1}.
\]
Using
\[
\bigl(T_m^{-1}\bigr)_{a,b}
=
\frac{\min(a,b)\bigl(m+1-\max(a,b)\bigr)}
{m+1},
\]
with
\[
m=q+1,\qquad a=j+1,\qquad b=q+1,
\]
we obtain
\[
\begin{aligned}
\bigl(T_{q+1}^{-1}\bigr)_{j+1,q+1}
&=
\frac{(j+1)\bigl(q+2-(q+1)\bigr)}
{q+2}=
\frac{j+1}{q+2}.
\end{aligned}
\]
Therefore, we have
\begin{equation}\label{3.2}
\bigl((A^n)^{-1}\bigr)_{l,n}=
\frac{j+1}{q+2}.
\end{equation}

On the other hand, by Theorem~\ref{t3.1},
\[
\det A^n=D_{n+1}.
\]
Since
\[
n+1=kq+r+1,
\]
the determinant formula gives
\begin{equation}\label{3.3}
D_{n+1}
=
(q+2)^{r+1}(q+1)^{k-r-1}.
\end{equation}

Combining \eqref{3.1}, \eqref{3.2}, and \eqref{3.3}, we obtain
\begin{equation}\label{3.4}
A^n_{n,l}=
D_{n+1}
\bigl((A^n)^{-1}\bigr)_{l,n}=
(q+2)^{r+1}(q+1)^{k-r-1}\frac{j+1}{q+2}=
(q+2)^r(q+1)^{k-r-1}(j+1).
\end{equation}

Since
\[
l=r+kj,
\]
we have
\[
j+1=\frac{k+l-r}{k}.
\]
Moreover, from \(n=kq+r\),
\[
q+2=\frac{n-r+2k}{k},
\qquad
q+1=\frac{n-r+k}{k}.
\]
Substituting these identities into \eqref{3.4}, we obtain
\[
\begin{aligned}
A^n_{n,l}
&=
\left(\frac{n-r+2k}{k}\right)^r
\left(\frac{n-r+k}{k}\right)^{k-r-1}
\frac{k+l-r}{k}\\
&=
\frac{
(n-r+2k)^r
(n-r+k)^{k-r-1}
(k+l-r)
}{k^k}.
\end{aligned}
\]
Therefore,
\[
A^n_{n,l}
=
\begin{cases}
\displaystyle
\frac{(n-r+2k)^r(n-r+k)^{k-r-1}(k+l-r)}
{k^k},
& l\equiv r\pmod{k},\\[3mm]
0,
& \text{otherwise}.
\end{cases}
\]
This completes the proof.
\end{proof}

\begin{remark}\label{r3.2}
By the centrosymmetry of $A^n$, we have
\begin{align*}
A_{0,j}^n=\left\lbrace
\begin{array}{ll}
\frac{(n-r+2k)^{r}(n-r+k)^{k-r-1}(n-j+k-r)}{k^k},~~ & n \equiv r~(\mathrm{mod}~k)~~~\mathrm{and}~~~j \equiv 0~(\mathrm{mod}~k),\\
0,~~ & \mathrm{otherwise},
\end{array}
\right.
\end{align*}
where $0\leq j\leq n, 0\leq r<k$, and $A_{0,j}^n$ represents the algebraic cofactor of the element in row 0 and column $j$.

\end{remark}

In general, determining the eigenvalues of core operators associated with arbitrary submodules is difficult. However, combining Lemma \ref{l3.3} with the preceding results, we can compute the eigenvalues of the core operator for the submodule $M_k$.

\begin{corollary}\label{c3.1}
Let $M_k=[z^k-w^k]$ with $k\geq 2$, and let $C$ be its core operator. The eigenvalues of the core operator $C$ are  
\[
\{0,1\}\cup \left\{\pm \frac{1}{m+1}: m\geq 1\right\},
\]  
where $0$ has infinite multiplicity, $1$ has multiplicity one, and each $\pm1/(m+1)$ has multiplicity one.

\end{corollary}

%Corollary 5.1. The core operators associated with the submodules $M_k=[z^k-w^k], k≥2$, are mutually unitarily equivalent.

\begin{proof}
By Lemma \ref{l3.3}, we know the core operator of submodule $[z^k-w^k]$ has eigenvalues
\[0,1,\pm \left(1-\frac{(|D_n|^2-|A_{0,n}^n|^2)^2}{D_{n-1}D_n^2D_{n+1}} \right)^{1/2},\,\,\,\,n \geq 1.\]

If $n=km, m \geq 1$, 
\[D_n=(m+1)^k,\,\,\,A_{0,n}^n=(m+1)^{k-1},\,\,\,D_{n-1}=m(m+1)^{k-1},\,\,\,D_{n+1}=(m+2)(m+1)^{k-1},\]
then
\[|D_n|^2-|A_{0,n}^n|^2=(D_n+A_{0,n}^n)(D_n-A_{0,n}^n)=m(m+2)(m+1)^{2k-2},\]
and
\[D_{n-1}D_n^2D_{n+1}=m(m+2)(m+1)^{4k-2}.\]

Therefore,
\[1-\frac{(|D_n|^2-|A_{0,n}^n|^2)^2}{D_{n-1}D_n^2D_{n+1}} =\frac{1}{(m+1)^2},\]
and the core operator has eigenvalues
\[\pm\frac{1}{m+1},\,\,\,m \geq 1.\]

If $n=km+l, 1 \leq l \leq k-1, m \geq 0$, then
\[D_n=(m+2)^l(m+1)^{k-l},\,\,\,A_{0,n}^n=0,\,\,\,D_{n-1}=(m+2)^{l-1}(m+1)^{k-l+1},\,\,\,D_{n+1}=(m+2)^{l+1}(m+1)^{k-l-1}.\]

Hence, it follows that
\[|D_n|^2-|A_{0,n}^n|^2=(D_n+A_{0,n}^n)(D_n-A_{0,n}^n)=(m+1)^{2k-2l}(m+2)^{2l};\]
\[D_{n-1}D_n^2D_{n+1}=(m+1)^{4k-4l}(m+2)^{4l}.\]

Therefore,
\[1-\frac{(|D_n|^2-|A_{0,n}^n|^2)^2}{D_{n-1}D_n^2D_{n+1}} =0,\]
and the corresponding eigenvalue is $0$.

In conclusion, for submodule $[z^k-w^k]$, the core operator has eigenvalues
\[0,\,\,\,1,\,\,\,\pm\frac{1}{m+1},\,\,\,m \geq 1.\]

For each $m\ge1$, the corresponding eigenspace is one-dimensional, since it arises from a single $n=km$ in Lemma \ref{l3.3}. The eigenvalue \(0\) has infinite multiplicity since
\(C\) vanishes on \(M_k^\perp\), which is infinite-dimensional, and
additional zero eigenvalues arise from the indices
\(n\not\equiv0\pmod k\). Since the eigenspace corresponding to the eigenvalue 1 is $(M_k\ominus zM_k)\bigcap (M_k\ominus wM_k)=\mathbb{C}(z^k-w^k)$,  the eigenvalue $1$ is simple.

% Using the result in~\cite{Zu2,Zou}, we obtain that the eigenvalue $1$ has multiplicity one:
% \[
% \begin{aligned}
% \operatorname{mult}(1)
% &=
% \|C\|_{H.S.}^2
% -
% 2\sum_{m=1}^{\infty}\frac1{(m+1)^2}\\
% &=
% \left(\frac{\pi^2}{3}-1\right)
% -
% 2\left(\frac{\pi^2}{6}-1\right)
% =1.
% \end{aligned}
% \]

\end{proof}

By Proposition 4.1 of \cite{Fat}, the second largest eigenvalue of the core operator is $\|[R_2^*, R_1]\|$. Therefore, for the submodule $M_k$, we have $\|[R_2^*, R_1]\|=\frac{1}{2}$. 

% Since the core operator associated with a homogeneous submodule is
% compact and self-adjoint, every nonzero point of its spectrum is an
% eigenvalue of finite multiplicity, and zero is the only possible
% accumulation point of the spectrum. By the preceding computation, the
% nonzero eigenvalues of $C$ are
% \[
% 1
% \quad\text{and}\quad
% \pm\frac{1}{m+1},\qquad m\geq 1,
% \]
% where $\sigma_p(C)$ denotes the point spectrum of $C$.

% Moreover, $0$ is an eigenvalue of infinite multiplicity. Therefore,
% there are no additional spectral points, and the full spectrum of $C$
% coincides with its point spectrum:
% \[
% \sigma(C)=\sigma_p(C)
% =
% \{0,1\}\cup
% \left\{
% \pm\frac{1}{m+1}:m\geq1
% \right\}.
% \]

% \begin{remark}
% When $k=1$, from p.517-518 of \cite{Fat}, we know that $D_n=n+1$ and $A_{0,n}^n=1$, therefore, we obtain
% \[1-\frac{(|D_n|^2-|A_{0,n}^n|^2)^2}{D_{n-1}D_n^2D_{n+1}} =\frac{1}{(n+1)^2}.\]

% Thus, for submodule $[z-w]$, the core operator has eigenvalues
% \[0,\,\,\,1,\,\,\,\pm\frac{1}{n+1},\,\,\,n \geq 1.\]
% \end{remark}

\section{Numerical invariants and monotonicity}\label{s4}
In this section, we investigate the numerical invariants and monotonicity of the submodules $M_k=[z^k-w^k]$ for $k\geq 2$. 

Throughout this section, set \(p=z^k-w^k\) and let
\(\{\varphi_n\}_{n\geq 0}\) and \(\{\psi_n\}_{n\geq 0}\) be the
orthonormal bases of \(M_k\ominus zM_k\) and \(M_k\ominus wM_k\),
respectively, given in Lemma~\ref{l3.1}.

We continue to use \(A^n_{i,j}\) to denote the algebraic cofactor of
the entry in row \(i\) and column \(j\) of the matrix \(A^n\).
Throughout this section, whenever one of the indices of \(A^n_{i,j}\)
lies outside \(\{0,1,\ldots,n\}\), the corresponding term is understood
to be zero.

For \(j\geq 1\), define
\[
\alpha_{n,j}:=\langle w^j\varphi_n,z^j\psi_n\rangle.
\]

We first derive a formula for \(\alpha_{n,j}\) that is valid for every
\(j\geq 1\). By Lemma~\ref{l3.1},
\[
\varphi_n
=
\frac{\displaystyle\sum_{a=0}^n
pA^n_{0,a}z^aw^{n-a}}
{\sqrt{D_nD_{n+1}}},
\qquad
\psi_n
=
\frac{\displaystyle\sum_{b=0}^n
pA^n_{n,b}z^bw^{n-b}}
{\sqrt{D_nD_{n+1}}}.
\]

Introduce the Toeplitz coefficients
\[
\gamma_t=
\begin{cases}
2, & t=0,\\
-1, & t=\pm k,\\
0, & \text{otherwise}.
\end{cases}
\]
Since \(p=z^k-w^k\), a direct calculation gives
\[
\left\langle
p z^aw^{n-a+j},
p z^{b+j}w^{n-b}
\right\rangle
=
\gamma_{j+b-a}.
\]
Consequently,
\begin{equation}\label{4.1}
\alpha_{n,j}
=
\frac{1}{D_nD_{n+1}}
\sum_{a,b=0}^n
A^n_{0,a}\gamma_{j+b-a}A^n_{n,b}.
\end{equation}

For fixed \(b\), put \(\ell=j+b\) and consider
\[
X_{\ell}^{(n)}
:=
\sum_{a=0}^n A^n_{0,a}\gamma_{\ell-a}.
\]
If \(1\leq \ell\leq n\), then
\((\gamma_{\ell-a})_{a=0}^n\) is the \(\ell\)-th row of \(A^n\).
Hence the cofactor expansion theorem gives
\[
X_{\ell}^{(n)}=0.
\]

If \(\ell=n+1\), then expanding the minor defining the cofactor \(A^{n+1}_{0,n+1}\) along its last row gives
\[
X_{n+1}^{(n)}=-A^{n+1}_{0,n+1}.
\]

If \(\ell=n+s\), where \(2\leq s\leq k\), then the only possible
nonzero Toeplitz coefficient is
\(\gamma_k=-1\), corresponding to \(a=n-k+s\). Therefore,
\[
X_{n+s}^{(n)}=-A^n_{0,n-k+s}.
\]

Finally, \(X_{\ell}^{(n)}=0\) whenever \(\ell\geq n+k+1\).
Substitution into \eqref{4.1} yields the unified
identity
\begin{equation}
\boxed{
\alpha_{n,j}
=
-\frac{1}{D_nD_{n+1}}
\left(
A^{n+1}_{0,n+1}A^n_{n,n-j+1}
+
\sum_{s=2}^{k}
A^n_{0,n-k+s}A^n_{n,n-j+s}
\right).
}
\label{eq:alpha-unified}
\end{equation}

% The convention on out-of-range indices makes
% \eqref{eq:alpha-unified} valid for all \(n\geq 0\) and \(j\geq 1\).
The unified formula (4.2) covers all cases of \(j\geq1\). 
When \(2\le j\le k\), some terms vanish automatically because the
corresponding cofactor indices fall outside the admissible range.
Therefore, no separate formula is required.
% In particular, when \(2\leq j\leq k\), the terms with \(s>j\) vanish,
% and \eqref{eq:alpha-unified} reduces to
% \[
% \alpha_{n,j}
% =
% -\frac{1}{D_nD_{n+1}}
% \left(
% A^{n+1}_{0,n+1}A^n_{n,n-j+1}
% +
% \sum_{s=2}^{j}
% A^n_{0,n-k+s}A^n_{n,n-j+s}
% \right).
% \]

We now determine the exact support and values of the inner products
\(\alpha_{n,j}\).

\begin{lemma}
\label{lem:alpha-support-values}
Let \(j\geq 1\), and write
\[
j=kq+r,\qquad 0\leq r<k.
\]
Then the following assertions hold.

\begin{enumerate}
\item If \(r=0\), then \(q\geq 1\), and
\[
\alpha_{k(q+t-1),j}
=
-\frac{t+1}{(q+t)(q+t+1)},
\qquad t\geq 0.
\]
Moreover, \(\alpha_{n,j}=0\) for all other \(n\).

\item If \(1\leq r<k\), then
\[
\alpha_{k(q+t+1)-r,j}
=
-\frac{t+1}{(q+t+1)(q+t+2)},
\qquad t\geq 0.
\]
Moreover, \(\alpha_{n,j}=0\) for all other \(n\).
\end{enumerate}
\end{lemma}

\begin{proof}
By Theorem~\ref{t3.2} and Remark~\ref{r3.2}, the cofactors satisfy
\begin{equation}
A^n_{n,l}\neq 0
\quad\Longrightarrow\quad
l\equiv n\pmod{k},
\label{eq:last-row-sparsity}
\end{equation}
and
\begin{equation}
A^n_{0,l}\neq 0
\quad\Longrightarrow\quad
l\equiv 0\pmod{k}.
\label{eq:first-row-sparsity}
\end{equation}

We first examine the special term in \eqref{eq:alpha-unified}.
For $A^{n+1}_{0,n+1}A^n_{n,n-j+1}$
to be nonzero, \eqref{eq:first-row-sparsity} and
\eqref{eq:last-row-sparsity} require $n+1\equiv 0\pmod{k}$ and $n-j+1\equiv n\pmod{k}$, respectively. Thus this term can be nonzero only when $j\equiv 1\pmod{k}$ and $n\equiv k-1\pmod{k}$.

Next consider the \(s\)-th term in the sum in
\eqref{eq:alpha-unified}. If
\[A^n_{0,n-k+s}A^n_{n,n-j+s}\neq 0,
\]
then we have $n-k+s\equiv 0\pmod{k}$ and $n-j+s\equiv n\pmod{k}$. Therefore, it yields that $s\equiv j\pmod{k}$ and
$n+j\equiv 0\pmod{k}$. Since \(2\leq s\leq k\), there is at most one admissible value of
\(s\): it is \(s=k\) when \(r=0\), and \(s=r\) when
\(2\leq r<k\). When \(r=1\), only the special term can be nonzero.

Suppose first that \(r=0\), so that \(j=kq\). Then only the term
\(s=k\) in \eqref{eq:alpha-unified} can be nonzero. The column index
\(n-j+k\) must be nonnegative, and hence the possible indices are
\[
n=k(q+t-1),\qquad t\geq 0.
\]
Set
\[
m=q+t-1.
\]
Then we have $n=km$ and $n-j+k=kt$.

By Theorems~\ref{t3.1} and~\ref{t3.2} and the centrosymmetry of \(A^n\), we have
\[A^n_{0,n}=(m+1)^{k-1},\,\,\,\,A^n_{n,kt}=(m+1)^{k-1}(t+1),
\]

and
\[
D_n=(m+1)^k,
\qquad
D_{n+1}=(m+2)(m+1)^{k-1}.
\]
Therefore,
\[
\begin{aligned}
\alpha_{n,j}
&=
-\frac{A^n_{0,n}A^n_{n,kt}}{D_nD_{n+1}}=
-\frac{
(m+1)^{k-1}(m+1)^{k-1}(t+1)
}{
(m+1)^k(m+2)(m+1)^{k-1}
}\\
&=
-\frac{t+1}{(m+1)(m+2)}=
-\frac{t+1}{(q+t)(q+t+1)}.
\end{aligned}
\]
This proves the first assertion.

Now suppose that \(r=1\), so that
\[
j=kq+1.
\]
Only the special term in \eqref{eq:alpha-unified} can be nonzero.
The possible indices are
\[
n=k(q+t+1)-1,\qquad t\geq 0.
\]
Set
\[
m=q+t.
\]
Then we have $n=km+k-1$ and $n-j+1=kt+k-1$.

By Theorems~\ref{t3.1} and~\ref{t3.2}, we have $A^{n+1}_{0,n+1}=(m+2)^{k-1},\,\,\,\,A^n_{n,n-j+1}=(m+2)^{k-1}(t+1)$, and
\[
D_n=(m+1)(m+2)^{k-1},
\qquad
D_{n+1}=(m+2)^k.
\]
Hence
\[
\begin{aligned}
\alpha_{n,j}
&=
-\frac{
A^{n+1}_{0,n+1}A^n_{n,n-j+1}
}{
D_nD_{n+1}
}=
-\frac{
(m+2)^{k-1}(m+2)^{k-1}(t+1)
}{
(m+1)(m+2)^{k-1}(m+2)^k
}\\
&=
-\frac{t+1}{(m+1)(m+2)}=
-\frac{t+1}{(q+t+1)(q+t+2)}.
\end{aligned}
\]

Finally, suppose that \(2\leq r<k\). Only the term \(s=r\) in
\eqref{eq:alpha-unified} can be nonzero. The possible indices are
\[
n=k(q+t+1)-r,\qquad t\geq 0.
\]
Set
\[
m=q+t,
\qquad
\ell=k-r.
\]
Then we have $n=km+\ell,\,\,\,\, n-k+r=km$, and
\[
n-j+r=\ell+kt.
\]

By Theorem~\ref{t3.2} and the centrosymmetry of \(A^n\),
\[
A^n_{0,n-k+r}
=
(m+2)^{k-r}(m+1)^{r-1},
\]
and
\[
A^n_{n,n-j+r}
=
(m+2)^{k-r}(m+1)^{r-1}(t+1).
\]
Theorem~\ref{t3.1} gives
\[
D_n=(m+2)^{k-r}(m+1)^r,
\]
and
\[
D_{n+1}
=
(m+2)^{k-r+1}(m+1)^{r-1}.
\]
It follows that
\[
\begin{aligned}
\alpha_{n,j}
&=
-\frac{
A^n_{0,n-k+r}A^n_{n,n-j+r}
}{
D_nD_{n+1}
}=
-\frac{
(m+2)^{2k-2r}(m+1)^{2r-2}(t+1)
}{
(m+2)^{2k-2r+1}(m+1)^{2r-1}
}\\
&=
-\frac{t+1}{(m+1)(m+2)}=
-\frac{t+1}{(q+t+1)(q+t+2)}.
\end{aligned}
\]
The sparsity conditions show that all remaining values of
\(\alpha_{n,j}\) vanish.
\end{proof}

We can now compute the complete sequence of numerical invariants.

\begin{theorem}
\label{thm:block-repetition}
Let \(M_k=[z^k-w^k]\), where \(k\geq 2\). Then we have $\Sigma_0(M_k)=\frac{\pi^2}{6}$ and
\begin{equation}
\boxed{
\Sigma_j(M_k)
=
\Sigma_{\lceil j/k\rceil}([z-w]),
\qquad j\geq 1.
}
\label{eq:block-repetition}
\end{equation}
Equivalently, for every \(m\geq 1\),
\[
\Sigma_{km-k+1}(M_k)
=
\Sigma_{km-k+2}(M_k)
=
\cdots
=
\Sigma_{km}(M_k)
=
\Sigma_m([z-w]).
\]
In particular, the sequence
\(\{\Sigma_j(M_k)\}_{j\geq 0}\) is nonincreasing.
\end{theorem}

\begin{proof}
By Lemma~\ref{l3.2},
\[
\Sigma_0(M_k)
=
\sum_{n=0}^{\infty}
\left|
\frac{A^n_{0,n}}{D_n}
\right|^2.
\]
The cofactor \(A^n_{0,n}\) vanishes unless
\[
n\equiv 0\pmod{k}.
\]
For \(n=km\), Theorems~\ref{t3.1} and~\ref{t3.2} give
\[
A^{km}_{0,km}=(m+1)^{k-1},
\qquad
D_{km}=(m+1)^k.
\]
Therefore, we have
\[
\Sigma_0(M_k)
=
\sum_{m=0}^{\infty}\frac{1}{(m+1)^2}
=
\frac{\pi^2}{6}.
\]

Let \(j\geq 1\), and write
\[
j=kq+r,\qquad 0\leq r<k.
\]
If \(r=0\), Lemma~\ref{lem:alpha-support-values} gives
\[
\begin{aligned}
\Sigma_j(M_k)=
\sum_{t=0}^{\infty}
\left|
\alpha_{k(q+t-1),j}
\right|^2=
\sum_{t=0}^{\infty}
\frac{(t+1)^2}
{(q+t)^2(q+t+1)^2}=
\Sigma_q([z-w]).
\end{aligned}
\]
Since
\[
\left\lceil\frac{j}{k}\right\rceil=q,
\]
this is
\[
\Sigma_j(M_k)
=
\Sigma_{\lceil j/k\rceil}([z-w]).
\]

If \(1\leq r<k\), Lemma~\ref{lem:alpha-support-values} gives
\[
\begin{aligned}
\Sigma_j(M_k)=
\sum_{t=0}^{\infty}
\left|
\alpha_{k(q+t+1)-r,j}
\right|^2=
\sum_{t=0}^{\infty}
\frac{(t+1)^2}
{(q+t+1)^2(q+t+2)^2}=
\Sigma_{q+1}([z-w]).
\end{aligned}
\]
Since
\[
\left\lceil\frac{j}{k}\right\rceil=q+1,
\]
formula \eqref{eq:block-repetition} follows.

Recall that
\[
\Sigma_m([z-w])
=
\sum_{t=0}^{\infty}
\frac{(t+1)^2}
{(m+t)^2(m+t+1)^2},
\qquad m\geq 1.
\]

For every fixed \(t\), the summand is strictly decreasing in \(m\).
Thus
\[
\{\Sigma_m([z-w])\}_{m\geq 1}
\]
is strictly decreasing. Therefore, formula \eqref{eq:block-repetition} 
shows that
\[
\{\Sigma_j(M_k)\}_{j\geq 0}
\]
is nonincreasing, with equality inside each block of length \(k\)
and strict decrease between consecutive blocks.

Moreover,
\[
\Sigma_0(M_k)=\frac{\pi^2}{6}
>
\frac{\pi^2}{6}-1
=
\Sigma_1(M_k).
\]
Therefore, the entire sequence
\(\{\Sigma_j(M_k)\}_{j\ge0}\) is nonincreasing.
\end{proof}

The block length in \eqref{eq:block-repetition} determines the
parameter \(k\). This gives a distinction between the higher
numerical invariants and the spectral data of the core operator.

\begin{corollary}
\label{cor:recover-k}
For \(k\geq 2\), the parameter \(k\) can be recovered from the
numerical invariant sequence of \(M_k\) by
\[
\boxed{
k
=
\max\left\{
j\geq 1:
\Sigma_j(M_k)=\Sigma_1(M_k)
\right\}.
}
\]
Consequently, if \(k\neq \ell\), then
\[
\{\Sigma_j(M_k)\}_{j\geq 0}
\neq
\{\Sigma_j(M_\ell)\}_{j\geq 0}.
\]
On the other hand, the nonzero eigenvalues of the corresponding core
operators, including multiplicities, are independent of \(k\).
Therefore, the higher numerical invariants detect information that is
not visible in the spectral data of the core operator.
\end{corollary}

\begin{proof}
By Theorem~\ref{thm:block-repetition},
\[
\Sigma_1(M_k)
=
\Sigma_2(M_k)
=
\cdots
=
\Sigma_k(M_k)
=
\Sigma_1([z-w]).
\]
On the other hand,
\[
\Sigma_{k+1}(M_k)
=
\Sigma_2([z-w])
<
\Sigma_1([z-w]).
\]
This proves the recovery formula and shows that distinct values of
\(k\) yield distinct numerical invariant sequences. On the other
hand, Corollary~\ref{c3.1} shows that the nonzero eigenvalues of the
corresponding core operators, including multiplicities, are
independent of \(k\).
\end{proof}

\begin{remark}
For \(k=2\), Theorem~\ref{thm:block-repetition} gives
\[
\Sigma_{2m-1}(M_2)
=
\Sigma_{2m}(M_2)
=
\Sigma_m([z-w]),
\qquad m\geq 1.
\]
Thus the separate computation for \([z^2-w^2]\) is recovered as the
first instance of the general block-repetition phenomenon. In
particular,
\[
\Sigma_1(M_2)
=
\Sigma_2(M_2)
=
\frac{\pi^2}{6}-1.
\]
\end{remark}

\medskip

\subsection*{Acknowledgment}
Y. Liu was supported by the Natural Science Foundation Project in Henan Province (No. 262300421861), the funding program for young backbone teachers in higher education institutions in Henan Province (No. 2024GGJS106), the key research projects of higher education institutions in Henan Province (No. 25B110009) and the general project cultivation fund of Nanyang Normal University (No. 2025PY034), Y. Lu was supported by NNSFC (Grant No. 12031002), C. Zu was supported by NNSFC (Grant No. 12401151), and the Postdoctoral Researcher Foundation of China  (Grant No. GZB20240100).

%\subsection*{Author Contribution Statement}
%All authors contributed equally to the research, analysis, and preparation of the manuscript. Each author approved the final version of the paper and agrees to be accountable for all aspects of the work.
%
\subsection*{Conflict of interest}
The authors declare that they have no conflict
of interest. 
\subsection*{Data availability statement}
No data, models, or code were generated or used for the research described in the article.%Data sharing is not applicable to this article as no new data were created or analyzed in this study.

\end{document}